\documentclass[12pt]{amsart}
\usepackage{amscd}
\usepackage{amssymb}

\def\frk{\frak}               

\def\pp{{\frk p}}

\def\mm{{\frk m}}

\def\Phi{{\frk n}}
\def\Phi{{\frk N}}
\def\opn#1#2{\def#1{\operatorname{#2}}} 
\opn\chara{char} \opn\length{\ell} \opn\pd{pd} \opn\rk{\lk}\opn\link{link}
\opn\projdim{proj\,dim} \opn\injdim{inj\,dim} \opn\rank{rank}\opn\Var{Var}
\opn\depth{depth} \opn\and{and} \opn\grade{grade}
\opn\height{height} \opn\embdim{emb\,dim} \opn\codim{codal}

\opn\Tr{Tr} \opn\bigrank{bigrank}
\opn\superheight{superheight}\opn\lcm{lcm}

\opn\trdeg{trdeg}
\opn\reg{reg} \opn\lreg{lreg} \opn\ini{in} \opn\mod{mod}
\opn\div{div} \opn\Div{Div} \opn\cl{cl} \opn\Cl{Cl}
\opn\Spec{Spec} \opn\Supp{Supp} \opn\supp{supp} \opn\Sing{Sing}
\opn\Ass{Ass} \opn\Min{Min} \opn\Var{Var}
\opn\Ann{Ann} \opn\Rad{Rad} \opn\Soc{Soc}
\opn\Im{Im} \opn\Spec{Spec
}
\opn\inf{inf}
 \opn\Ker{Ker} \opn\Coker{Coker} \opn\Am{Am} \opn \inf{inf}
\opn\Hom{Hom} \opn\Tor{Tor} \opn\Ext{Ext} \opn\End{End} \opn\cd{cd}
\opn\Aut{Aut} \opn\id{id}

\opn\nat{nat}
\opn\pff{pf}
\opn\Pf{Pf} \opn\GL{GL} \opn\SL{SL} \opn\mod{mod} \opn\ord{ord}
\opn\cl{cl} \opn\conv{conv} \opn\ext{ext} \opn\rad{rad}
\opn\star{star} \opn\red{red}\opn\H{H} \opn\bight{bight}
\opn\aff{aff} \opn\con{conv} \opn\relint{relint} \opn\st{st}
\opn\lk{lk} \opn\cn{cn} \opn\core{core} \opn\vol{vol}
\opn\link{link} \opn\star{star}  \opn\mdepth{mdepth}
\opn\gr{gr}

\def\pot#1#2{#1[\kern-0.28ex[#2]\kern-0.28ex]}

\opn\dirlim{\underrightarrow{\lim}}
\opn\inivlim{\underleftarrow{\lim}}
\let\union=\cup
\let\sect=\cap
\let\dirsum=\oplus

\let\iso=\cong

\let\Sect=\bigcap

\let\to=\rightarrow

\def\Implies{\ifmmode\Longrightarrow \else
     \unskip${}\Longrightarrow{}$\ignorespaces\fi}
\def\implies{\ifmmode\Rightarrow \else
     \unskip${}\Rightarrow{}$\ignorespaces\fi}
\def\iff{\ifmmode\Longleftrightarrow \else
     \unskip${}\Longleftrightarrow{}$\ignorespaces\fi}

\let\:=\colon
\newtheorem{Theorem}{Theorem}[section]
\newtheorem{Lemma}[Theorem]{Lemma}
\newtheorem{Corollary}[Theorem]{Corollary}
\newtheorem{Proposition}[Theorem]{Proposition}
\newtheorem{Remark}[Theorem]{Remark}

\newtheorem{Example}[Theorem]{Example}

\newtheorem{Definition}[Theorem]{Definition}

\newtheorem{Fact}[Theorem]{Fact}

\let\epsilon\varepsilon
\let\phi=\varphi
\let\kappa=\varkappa
\def\qed{\ifhmode\textqed\fi
   \ifmmode\ifinner\quad\qedsymbol\else\dispqed\fi\fi}
\def\textqed{\unskip\nobreak\penalty50
    \hskip2em\hbox{}\nobreak\hfil\qedsymbol
    \parfillskip=0pt \finalhyphendemerits=0}
\def\dispqed{\rlap{\qquad\qedsymbol}}

\opn\dis{dis}
\def\pnt{{\raise0.5mm\hbox{\large\bf.}}}

\begin{document}

\title{Squarefree monomial ideals with maximal depth}

\author{ Ahad Rahimi}

.

\address{ Ahad Rahimi, Department of Mathematics, Razi University, Kermanshah, Iran}\email{ahad.rahimi@razi.ac.ir}

\begin{abstract}
Let $(R,\mm)$ be a Noetherian local ring and $M$ a finitely generated $R$-module.  We say $M$ has maximal depth if there is an associated prime $\pp$ of $M$  such that $\depth M=\dim R/\pp$. In this paper we study squarefree monomial ideals which have maximal depth. Edge ideals of cycle graphs, transversal polymatroidal ideals and high powers of connected bipartite graph with this property are classified.
\end{abstract}

\maketitle

\section*{Introduction}
Let $K$ be a field and $(R,\mm)$ be a Noetherian local ring, or a standard graded $K$-algebra with graded
maximal ideal $\mm$. Let $M$ be a finitely generated $R$-module.
A basic fact in commutative algebra says that
\[
\depth M\leq \min\{\dim R/\pp: \pp \in \Ass(M)\}.
\]
We set $\mdepth_R M=\min\{\dim R/\pp: \pp \in \Ass(M)\}$. For simplicity, we write $\mdepth M$ instead of $\mdepth_RM$. We say $M$ has {\em maximal depth} if the equality holds, i.e., $\depth M =\mdepth M.$ In other words, there is an associated prime $\pp$ of $M$  such that $\depth M=\dim R/\pp$. In this paper, we study squarefree monomial ideals with maximal depth.

Let $I\subset S=K[x_1,\ldots,x_n]$ be a squarefree monomial ideal. We say $I$ has maximal depth if $S/I$ has maximal depth.  We observe that, $I$ has maximal depth is equivalent to say that $\reg I^\vee$ is the maximum degree of the generators of $I^\vee$.  This fact motivates us to work on squarefree monomial ideals with maximal depth. Here $I^\vee$ is the {\em Alexander dual} of $I$ and $\reg M$ denotes the regularity
of a finitely generated graded $S$-module $M$.

Several authors have been working on this topic and some known results in this regards are as follows: If $I \subset S$ is a generic monomial ideal, then it has maximal depth, see \cite[Theorem 2.2]{MSY}. If a monomial ideal $I$ has maximal depth, then so does its polarization, see \cite{FT}. Algebraic properties and some classifications of modules with maximal depth are given in \cite{R}.

In \cite{SJ},  the depth of the line graph $L_n$ is explicitly computed. In Section 2, we compute the depth of the line graph $L_n$ in a different way. Our proof relies on the fact that trees, and line graphs in particular, have maximal depth, see Proposition~\ref{line}.

In \cite{SJ},  the depth of the cycle graph $C_n$ of length $n$ is also computed.  This number is independent
of the characteristic of the chosen field. By using this result, we classify all cycle graphs $C_n$ which have maximal depth. In fact, $C_n$ has maximal depth if and only if $n\equiv 0(\mod 3)$ or $n\equiv 2(\mod 3)$, see Proposition~\ref{cycle2}.

Adding a whisker to $C_n$ at a vertex $x_1$  means adding a new vertex
$x_{n+1}$  and the edge $\{x_1, x_{n+1}\}$ to $C_n$.   We denote by $C_n\cup W(x_1)$  the graph
obtained from $C_n$  by adding a whisker at $x_1$.  By using Proposition \ref{line}, we show that  $C_n$ and $C_n\cup W(x_1)$ have the same depth as well as  $C_n\cup W(x_1)$ has maximal depth.


In Section 3, we consider the transversal polymatroidal ideals. A transversal polymatroidal ideal is an ideal $I$ of the form
$I = \pp_{F_1}\dots\pp_{F_r}$  where $F_1, \dots, F_r$ is a collection of non-empty subsets of $[n]$ with $r\geq 1$. Here for a non-empty subset $F$  of $[n]$, we denote by $\pp_F$ the monomial prime
ideal $(\{x_i: i \in F\})$. The depth of a transversal polymatroidal ideal is explicitly given in \cite{HRV}. By applying this result, we classify all transversal polymatroidal ideals which have maximal depth. In fact, we prove the following: Let $I\subset S$ be a transversal polymatroidal ideal. Then,  $I$ has maximal depth if and only  $I$ is a product of monomial prime ideals such that at most one of the factors is not principal. In the following, we also classify ideals of Veronese type which have maximal depth.

In the final section, we consider $G$ to be a connected bipartite graph and $I$ its edge ideal. We show $I^k$ has maximal depth for $k\gg 0$ if and only if $G$ is a star graph.

\section{Preliminaries}

Let $K$ be a field and $(R,\mm)$ be a Noetherian local ring, or a standard graded $K$-algebra with graded
maximal ideal $\mm$.
It is a classical fact that if $M\neq 0$ is an $R$-module, then
\[
\depth M\leq \min\{\dim R/\pp: \pp \in \Ass(M)\},
\]
see \cite{BH}. We set $\mdepth_R M=\min\{\dim R/\pp: \pp \in \Ass(M)\}$. For simplicity, we write $\mdepth M$ instead of $\mdepth_RM$.
Thus $\depth M\leq \mdepth M\leq \dim M$. Observe that $\depth(M)=0$ if and only if $\mdepth(M)=0$. Thus, if $\mdepth M=1$, then $\depth M=1$.
\begin{Definition}{\em
We say $M$ has {\em maximal depth} if the equality holds, i.e.,
\[
\depth M =\mdepth M.
\]
In other words, there is an associated prime $\pp$ of $M$  such that $\depth M=\dim R/\pp$.}
\end{Definition}
Some examples of modules with maximal depth property are as follows:
 \begin{itemize}
\item Cohen--Macaulay modules have maximal depth because $\depth M=\dim R/\pp$ for every associated prime of $M$, see \cite[Proposition 2.3.13]{V}.
\item Sequentially Cohen-Macaulay modules have maximal depth, see \cite[Proposition 1.4]{R}, see also \cite[Theorem 6.4.23]{V} where the ring $R$ is a polynomial ring.
\item If $M$ is unmixed, then $M$ has maximal depth if and only if $M$ is Cohen--Macaulay.
\end{itemize}

Let $I\subset S=K[x_1,\ldots,x_n]$ be a squarefree monomial ideal. Then $I=\Sect_{j=1}^m \pp_{j}$ where each of the $\pp_j$  is a monomial prime ideal of $I$.  The ideal $I^\vee$ which is minimally generated by the monomials $u_{j}=\prod_{x_i\in \pp_j}x_i$ is called the {\em Alexander dual} of $I$. As usual we denote by $\reg M$ the regularity
of a finitely generated graded $S$-module $M$. We quote the following facts which for example can be found in \cite{HH}.

\begin{Theorem}
\label{Terai}
(Terai) $\reg I^\vee=\pd S/I$.
\end{Theorem}


\begin{Theorem}
\label{Aus}
(Auslander-Buchsbam formula) Let $M$ be a finitely generated $R$-module with $\pd M<\infty$. Then
\[
\pd M+\depth M=\depth R.
\]
\end{Theorem}
The {\em big height} of an ideal $J\subset S$, denoted by $\bight J$, is the maximum height of the minimal primes of $J$.
The following simple fact motivates us to work on squarefree monomial ideals with maximal depth. We say $I\subset S$ has maximal depth if $S/I$ has maximal depth.
\begin{Proposition}
\label{Teri}
Let $I\subset S$ be a squarefree monomial ideal. Then, $I$ has maximal depth if and only if $\reg I^\vee$ is the maximum degree of the generators of $I^\vee$.
\end{Proposition}
\begin{proof}

Suppose $I$ has maximal depth. Hence
\begin{eqnarray*}
\reg I^\vee&=&\pd S/I\\
           &=&n-\depth S/I\\
           &=&n-\mdepth S/I\\
           &=& \bight I.
\end{eqnarray*}
Theorem \ref{Terai}  explains the first step in this sequence. Theorem \ref{Aus} provides the second step. Our assumption implies the third step. The forth step follows from that fact that when $I$ is squarefree, the associated primes are the same as minimal primes containing $I$.  Notice that the $\bight I$ is the maximum degree of the generators of $I^\vee$. Therefore, the conclusion follows.
Conversely, suppose $\reg I^\vee$ is the maximum degree of the generators of $I^\vee$. By the same reasons as above, we have
\[
\depth S/I=n-\pd S/I
          =n-\reg I^\vee
          =n-\bight I
          =\mdepth S/I,
          \]
as desired.
\end{proof}
We recall the following fact from \cite[Lemma 2.3.8]{V}.
\begin{Lemma}
\label{Depth lemma}(Depth Lemma) If  $0\to N \to M \to L \to 0$ is a short exact sequence of $R$-modules, then
 \begin{itemize}
\item[{(a)}] If $\depth(M)<\depth(L)$, then $\depth(N)=\depth(M)$.
\item[{(b)}] If $\depth(M)=\depth(L)$, then $\depth(N)\geq \depth(M)$.
\item[{(c)}] If $\depth(M)>\depth(L)$, then $\depth(N)=\depth(L)+1$.
\end{itemize}
\end{Lemma}

\section{Line and Cycle graphs}

Let $G$ be a graph.  The vertex set of $G$ will be denoted by $V(G)$ and will be the set $[n]=\{1,2,\ldots,n\}$. We denote the set of edges of $G$ by $E(G)$. We consider the {\em edge ideal} $I(G)$  which is generated by all monomials $x_ix_j$ with $\{i,j\}\in E(G)$. A subset $C\subset [n]$ is called a {\em vertex cover} of $G$  if $C\sect\{i,j\}\neq \emptyset$ for all edges $\{i,j\}$ of $G$.
A vertex cover $C$ is called {\em minimal} if $C$ is a
vertex of $G$, and no proper subset of $C$ is a vertex cover of $G$. A minimal vertex cover of $G$ is
called {\em maximum} if it has maximum cardinality among the minimal vertex covers of $G$. Thus $\bight I(G)$ is the cardinality of the maximum minimal vertex covers of $G$.

It is well known that the minimal vertex covers of $G$ are the sets of generators of the minimal primes of $I(G)$. In fact, a subset $C=\{ i_1, \dots, i_r\}\subset [n]$  is a minimal vertex cover of $G$ if and only if $\pp_C=(x_{i_1}, \dots, x_{i_r})$ is a
minimal prime ideal of $I(G)$, see \cite[Lemma 9.1.4]{HH}.

The graph  $G$ is called {\em disconnected} if $V(G)$ is the disjoint union of $W_1$ and $W_2$ and there is no edge $\{i,j\}$ of $G$ with $i\in W_1$ and $j\in W_2$. The graph $G$ is called {\em connected} if it is not disconnected.  A graph which has no cycle and which is connected is called a {\em tree}.

For $n\geq 2$, we let $L_n$ denote the line graph on $n$ vertices. This is the graph with vertices $[n]$ and edges $\{j, j + 1\}$ for all $j = 1, \dots, n-1$. Hence, its edge ideal is $I(L_n)=(x_1x_2, x_2x_3, \dots, x_{n-1}x_n)$ in a polynomial ring with $n$ variables.
In the following, we explicitly compute the depth of the line graph $L_n$. However, this is a known fact, see \cite[Corollary 7.7.35]{SJ} but here we prove it in a different way.
\medskip
\;\;\;

\noindent\textbf{Notation}:
 For any graph $G$, we write $\depth G$ for the depth of $S/I(G)$.

\begin{Proposition}
\label{line}
The depth of the line graph $L_n$ is independent
of the characteristic of the chosen field and is
\[
\depth L_n= \left\{
\begin{array}{ccc}
\frac{n}{3} &&\text{if \; $n\equiv 0(\mod 3)$},\\[1ex]
\frac{n+2}{3} && \text{if  \; $n\equiv 1(\mod 3)$,}\\[1ex]
\frac{n+1}{3} && \text{if  \; $n\equiv 2(\mod 3)$.}
\end{array}
\right.
 \]
\end{Proposition}
\begin{proof}
Notice that the line graph is a tree.  Trees are sequentially Cohen-Macaulay, see \cite{F}. As sequentially Cohen-Macaulay modules have maximal depth, all trees have maximal depth. In particular, $L_n$ has maximal depth for all $n$.  Let $I=I(L_n)$ be the edge ideal of $L_n$ in a polynomial ring $S$ with $n$ variables. We consider the following cases:

Case 1: $n\equiv 0(\mod 3)$.  We claim that the set
 \[
C=\{1, 3, 4, 6,  7, 9, 10, \dots, n-3, n-2, n\}
 \]
  is a maximum minimal vertex cover of $L_n$. A minimal vertex cover of $L_n$ cannot contain $3$ consecutive vertices because of minimality. This implies that if we divide the vertices of $L_n$ into blocks of $3$ vertices, then each block can have at most $2$ vertices in the cover. Therefore the cardinality of a minimal vertex cover can be at most $2n/3$.  Thus
  \[
  \pp_C=(x_1, x_3, x_4, x_6,  x_7, x_9, x_{10},\dots, x_{n-3}, x_{n-2}, x_n)
  \]
  is a minimal prime ideal of $I$ with maximum height and so $\bight I= 2n/3$. It follows that $\mdepth L_n=n-2n/3=n/3$ and hence $\depth L_n=n/3$.

  Case 2:  $n\equiv 1(\mod 3)$.  Hence $n-1\equiv 0(\mod 3)$. We claim that the set
   \[
C=\{ 2, 3, 5, 6, 8, 9,\dots, n-2, n-1\}
 \]
is a maximum minimal vertex cover of $L_n$. In fact, the vertices of $L_n$ can be divided into blocks with $3$ vertices as well as one block with only $1$ vertex. Then each block of $3$ vertices can have at most $2$ vertices in the cover. The vertex in the block with one vertex need not to be in the cover. Therefore the cardinality of a minimal vertex cover can be at most $2(n-1)/3$. Hence
\[
\pp_{C}=(x_2, x_3, x_5, x_6, x_8, x_9,\dots, x_{n-2}, x_{n-1})
 \]
 is a minimal prime ideal of $I$ with maximum height and so $\bight I= 2(n-1)/3$. Consequently, $\mdepth L_n=n-2(n-1)/3=(n+2)/3$ and so $\depth L_n=(n+2)/3$.

  Case 3:  $n\equiv 2(\mod 3)$.  Hence, $n-2\equiv 0(\mod 3)$.  The set
\[
C=\{ 2, 3, 5, 6, 8, 9,\dots, n-3, n-2, n \}
 \]
is a maximum minimal vertex cover of $L_n$. Indeed, the vertices of $L_n$ can be divided into blocks with $3$ vertices as well as only one block with $2$ vertex. Then each block of $3$ vertices can have at most $2$ vertices and the block of $2$ vertices can have at most $1$ vertex in the cover.  Therefore the cardinality of a minimal vertex cover can be at most $2(n-2)/3+1=(2n-1)/3$. Hence
\[
\pp_{C}=(x_2, x_3, x_5, x_6, x_8, x_9,\dots, x_{n-3}, x_{n-2}, x_n)
\]
is a minimal prime ideal of $I$ with maximum height and so $\bight I= 2(n-1)/3$. Consequently,  $\mdepth L_n=n-(2n-1)/3=(n+1)/3$ and so $ \depth L_n=(n+1)/3.$ We remark that the proof of proposition does not depend on the characteristic of the field $K$.
\end{proof}

Let $C_n$ be a cycle graph of length $n$.
 We recall the following result from \cite[Corollary 7.6.30]{SJ}.

\begin{Fact}{\em
\label{cycle1}
The depth of the cycle graph is independent
of the characteristic of the chosen field and is
\[
\depth C_n= \left\{
\begin{array}{ccc}
\frac{n}{3} &&\text{if \; $n\equiv 0(\mod 3)$},\\[1ex]
\frac{n-1}{3} && \text{if  \; $n\equiv 1(\mod 3)$,}\\[1ex]
\frac{n+1}{3} && \text{if  \; $n\equiv 2(\mod 3)$.}
\end{array}
\right.
 \]}
 \end{Fact}
In the following, we classify all cycle graphs which have maximal depth.
\begin{Proposition}
\label{cycle2}
The cycle graph $C_n$ has maximal depth if and only if $n\equiv 0(\mod 3)$ or $n\equiv 2(\mod 3)$.
\end{Proposition}
\begin{proof}
Let $I=I(C_n)$ be the edge ideal of $C_n$ in a polynomial ring $S$ with $n$ variables. We need to consider the following three cases.

Case 1:
 $n\equiv 0(\mod 3)$. For the maximum
minimal vertex covers of cycles one can use the line graphs. A similar argument as in the proof of Proposition \ref{cycle2} shows that,  the cardinality of a minimal vertex cover of $C_n$ in this case can be at most $2n/3$. The set
 \[
C=\{i, i+1, i+3, i+4,  i+6, i+7,\dots, n-i-1, n-i \}
 \]
 is a maximum minimal vertex cover of $C_n$ for all $i$. Thus,
  \[
 \pp_C=(x_i, x_{i+1}, x_{i+3}, x_{i+4},  x_{i+6}, x_{i+7},\dots, x_{n-i-1}, x_{n-i})
  \]
  is a minimal prime ideal of $I$ with maximum height and so $\bight I=2n/3$. Hence $\mdepth C_n=n-2n/3=n/3=\depth C_n.$ Fact \ref{cycle1} provides the last equality. Therefore, $C_n$ has maximal depth.

Case 2:   $n\equiv 2(\mod 3)$.  A similar argument as in the proof of Proposition \ref{cycle2} shows that,  the cardinality of a minimal vertex cover of $C_n$ in this case can be at most $2(n-2)/3+1=(2n-1)/3$. One observes that the set
\[
 C= \{ i, i+1, i+3, i+4, \dots, n+i-5, n+i-4, n+i-2\}
 \]
is a maximum minimal vertex cover of $C_n$ for all $i$. Hence
\[
\pp_{C}=(x_i, x_{i+1}, x_{i+3}, x_{i+4}, \dots, x_{n+i-5}, x_{n+i-4}, x_{n+i-2})
 \]
 is a minimal prime ideal of $I$ with maximum height and so $\bight I=(2n-1)/3$. Consequently,
  \[
  \mdepth C_n=n-(2n-1)/3=(n+1)/3=\depth C_n.
  \]
   Fact \ref{cycle1} explains the last equality.

Case 3:  $n\equiv 1(\mod 3)$.  In this case, one has that, the cardinality of a minimal vertex cover of $C_n$ can be at most $2(n-1)/3$ and the set
 \[
C=\{ i, i+2, i+3, i+5, i+6, \dots, n+i-5, n+i-4, n+i-2\}
 \]
is a maximum minimal vertex cover of $C_n$ for all $i$. Hence
\[
\pp_{C}=(x_i, x_{i+2}, x_{i+3}, x_{i+5}, x_{i+6}, \dots, x_{n+i-5}, x_{n+i-4}, x_{n+i-2})
\]
 is a minimal prime ideal of $I$ with maximum height and so $\bight I=2(n-1)/3$. Thus $\mdepth C_n=n-2(n-1)/3=(n+2)/3$. Fact \ref{cycle1} provides $\depth C_n=(n-1)/3$. Thus, $C_n$ has no maximal depth in this case.
\end{proof}


 Adding a whisker to $C_n$ at a vertex $x_1$  means adding a new vertex
$x_{n+1}$  and the edge $\{x_1, x_{n+1}\}$ to $C_n$.   We denote by $C_n\cup W(x_1)$  the graph
obtained from $C_n$  by adding a whisker at $x_1$.  Thus $I(C_n\cup W(x_1))=I(C_n)+(x_1x_{n+1})$.
 In  the following, by using Proposition \ref{line}, we show that  $C_n$ and $C_n\cup W(x_1)$ have the same depth as well as  $C_n\cup W(x_1)$ has maximal depth.
\begin{Proposition}
\label{whisker}
The following statements hold.
\[
\depth C_n=\depth  C_n\cup W(x_1),
\]
 and $ C_n\cup W(x_1)$ has maximal depth.
\end{Proposition}
\begin{proof}
We set $I(C_n)=J$ and $I(C_n\cup W(x_1))=I$. Consider the exact sequence
\begin{equation}
\label{exact}
0\to R/(I: x_{n+1}) \to R/I \to R/(I+(x_{n+1})) \to 0,
\end{equation}
where $R=S[x_{n+1}]$.  One has
\[
 R/(I: x_{n+1}) \iso K[x_2, \dots, x_n][x_{n+1}]/(x_2x_3, x_3x_4, \dots, x_{n-1}x_n),
\]
and
\[
R/(I+(x_{n+1})) \iso S/J.
\]
We consider the following three cases:

Case 1: $n\equiv 0(\mod 3)$.  Thus $n-1\equiv 2(\mod 3)$. By Proposition \ref{line}
\[
 \depth K[x_2, \dots, x_n]/(x_2x_3, x_3x_4, \dots, x_{n-1}x_n)=((n-1)+1)/3=n/3.
\]
Hence $\depth  R/(I: x_{n+1})=n/3+1$.
Fact \ref{cycle1} provides $\depth S/J=n/3$. Thus by using (1) we have
\[
\depth R/I \geq \min \{n/3+1, n/3 \}=n/3.
\]
  For computing $\mdepth R/I$ in this case, a similar argument as in the proof of Proposition \ref{cycle2} shows that,  the cardinality of a minimal vertex cover of $C_n\cup W(x_1)$ can be at most $(2n+3)/3$. One observes that the set
\[
C=\{n+1, 2, 3, 5, 6, \dots, n-4, n-3, n-1, n\}
\]
is a maximum minimal vertex cover of $C_n\cup W(x_1)$. Thus
\[
\pp_C=(x_{n+1}, x_2, x_3, x_5, x_6, \dots, x_{n-4}, x_{n-3}, x_{n-1}, x_n)
\]
 is a minimal prime ideal of $I$ with $\bight \pp_C=(2n+3)/3$. Hence
\[
\mdepth R/I=(n+1)-\bight \pp_C=(n+1)-(2n+3)/3=n/3.
\]
Consequently,
\[
n/3 \leq \depth R/I\leq \mdepth R/I=n/3.
\]
 Thus,  the results follow in this case.

Case 2: $n\equiv 1(\mod 3)$.  Thus $n-1\equiv 0(\mod 3)$. By Proposition \ref{line}
\[
 \depth K[x_2, \dots, x_n]/(x_2x_3, x_3x_4, \dots, x_{n-1}x_n)=(n-1)/3.
\]
Hence $\depth  R/(I: x_{n+1})=(n+2)/3$.
  Fact \ref{cycle1} explains $\depth S/J=(n-1)/3$. Thus by using (1) we have
  \[
   \depth R/I \geq \min \{(n+2)/3, (n-1)/3 \}=(n-1)/3.
   \]
 One observes that the cardinality of a minimal vertex cover of $C_n\cup W(x_1)$ in this case can be at most $2(n+2)/3$ and the set
 \[
C=\{n+1, 2, 3, 5, 6, \dots, n-5, n-4, n-2, n\}
\]
is a maximum minimal vertex cover of $C_n\cup W(x_1)$. Thus
\[
\pp_{C}=(x_{n+1}, x_2, x_3, x_5, x_6, \dots, x_{n-5}, x_{n-4}, x_{n-2}, x_n)
 \]
 is a minimal prime ideal of $I$ with $\bight \pp_{C}=2(n+2)/3$. Hence
 \[
 \mdepth R/I=(n+1)-(2n+4)/3=(n-1)/3.
 \]
 We conclude that
 \[
 (n-1)/3 \leq \depth R/I\leq \mdepth R/I=(n-1)/3.
 \]
 Therefore, the desired conclusions follow in this case.

Case 3: $n\equiv 2(\mod 3)$.  Thus $n-1\equiv 1(\mod 3)$. By Proposition \ref{line}
\[
 \depth K[x_2, \dots, x_n]/(x_2x_3, x_3x_4, \dots, x_{n-1}x_n)=((n-1)+2)/3=(n+1)/3.
\]
Hence $\depth  R/(I: x_{n+1})=(n+4)/3$.
 Fact \ref{cycle1} provides  $\depth S/J=(n+1)/3$. Thus by using (1) we have
 \[
 \depth R/I\geq \min \{(n+4)/3, (n+1)/3 \}=(n+1)/3.
 \]
 One observes that the cardinality of a minimal vertex cover of $C_n\cup W(x_1)$ can be at most $2(n+1)/3$ and the set
 \[
C=\{n+1, 2, 3, 5, 6, \dots, n-3, n-2, n\}
\]
is a maximum minimal vertex cover of $C_n\cup W(x_1)$. Thus
\[
\pp_{C}=(x_{n+1}, x_2, x_3, x_5, x_6, \dots,  x_{n-3}, x_{n-2}, x_n)\]
 is a minimal prime ideal  of $I$ with $\bight \pp_{C}=2(n+1)/3$. Hence
 \[
 \mdepth R/I=(n+1)-\bight \pp_{C}=(n+1)-(2n+2)/3=(n+1)/3.
 \]

Consequently,
\[
(n+1)/3 \leq \depth R/I\leq \mdepth R/I=(n+1)/3.
\]
 Therefore, the desired conclusions follow in this case too.
\end{proof}



We remark that the second part of Proposition \ref{whisker} also follows from \cite[Corollary 3.4]{FH} in a different way.


\section{Transversal polymatroids and Ideals of Veronese type}
In this section, we classify all transversal polymatroidal ideals and all ideals of Veronese type which have maximal depth.
Let $F$ be a non-empty subset of $[n]$. We denote by $\pp_F$ the monomial prime
ideal $(\{x_i: i \in F\})$. A {\em transversal polymatroidal ideal} is an ideal $I$ of the form
$I = \pp_{F_1}\dots\pp_{F_r}$  where $F_1, \dots, F_r$ is a collection of non-empty subsets of $[n]$ with $r\geq 1$.
Let $G_I$ be the graph with vertex set $\{1, \dots, r\}$ and for which $\{i, j\}$  is an edge of $G_I$ if and only if $F_i\cap F_j \neq\emptyset$.
We recall the following fact from \cite[Theorem 4.12]{HRV}.
\begin{Fact}{\em
\label{depth}
 Let $I=\pp_{F_1}\dots\pp_{F_r}\subset S$ be a transversal polymatroidal ideal. Then
\[
\depth S/I=c(G_I)-1+n-|\cup_{i=1}^r F_i\mid,
\]
where by $c(G_I)$ we denote the number of connected components of the graph $G_I$.}
\end{Fact}
Let $\mathcal{H}$ be a subgraph of $G_I$. We associate the prime ideal $\pp_\mathcal{H}=\sum_{i\in \mathcal{V}(\mathcal{H})}\pp_{F_i}$. We denote by $\Ass(I)$  the set
of associated prime ideals of $R/I$. The set associated primes of $R/I$ is explicitly described in \cite[Theorem 4.7]{HRV} as follows.

\begin{Fact}{\em
\label{asstree}
Let $I\subset S$ be a transversal polymatroidal ideal. Then
\[
\Ass(I)=\{ \pp_\mathcal{T}: \mathcal{T} \quad \text {is a tree in} \;\; G_I\}.
\]}
\end{Fact}
In the following we characterize all transversal polymatroidal ideals which have maximal depth.
\begin{Proposition}
Let $I=\pp_{F_1}\dots\pp_{F_r}\subset S$ be a transversal polymatroidal ideal. The following conditions are equivalent:
 \begin{itemize}
\item[{(a)}]  $I$ has maximal depth;
\item[{(b)}] $I$ is a product of monomial prime ideals such that at most one of the factors is not principal.
\end{itemize}
\end{Proposition}
\begin{proof}
(a)\implies (b):
We may assume that $\cup_{i=1}^r F_i=[n]$. Let $k=c(G_I)$ and $G_1, \dots, G_k$ be the connected components of $G_I$. Fact \ref{depth} provides
$\depth S/I=k-1$. We denote by $I_1, \dots, I_k$  the transversal polymatroidal ideals for which the associated graphs are the connected components of $G_I$. Hence $I=I_1\dots I_k=I_1\cap \dots \cap I_k$, since
the ideals $I_j$ are generated in pairwise disjoint sets of variables. Thus $\Ass(I)=\bigcup_{i=1}^k \Ass(I_i)$.  We may assume that $1\leq l_1\leq \dots \leq l_k$ where for all $j$ we have $l_j=|\bigcup_{i\in \mathcal{V}(G_j)}F_i| $. Note that $l_1+ \dots + \l_k=n$. In view of Fact \ref{asstree}, we have $\mdepth S/I=n-l_k$.  Since $I$ has maximal depth, it follows that $n-l_k=k-1$, and hence $l_1+\dots +l_{k-1}=k-1$. Consequently,  $I=(x_1)\dots (x_{k-1})(x_k, \dots, x_n)$, as desired.

(b)\implies (a): If $I$ is a product of monomial prime ideals such that all the factors are principal, then $S/I$ is Cohen--Macaulay and hence $I$ has maximal depth. Thus, we may assume that  $I=(x_1)\dots (x_{k-1})(x_k, \dots, x_n)$.  As \[
\Ass(I)=\{ (x_1), \dots,  (x_{k-1}), (x_k, \dots, x_n)\},
 \]
 we have $\mdepth S/I=n-(n-k+1)=k-1$. The ideal $I$ is a transversal polymatroidal ideal. It follows from Theorem \ref{depth} that $\depth S/I=k-1$. Here $c(G_I)=k$ and $|\cup_{i=1}^r F_i\mid=n$. Therefore, $I$ has maximal depth.
\end{proof}
As a consequence one has
\begin{Corollary}
Let $I\subset S$ be the intersection of monomial prime ideals in pairwise disjoint sets of variables.  Then $I$ has maximal depth if and only if $I$ is a product of monomial prime ideals such that at most one of the factors is not principal.
\end{Corollary}

One of the most distinguished polymatroidal ideals is the ideal of Veronese type. Let
$S = K[x_1, \dots, x_n]$ and fix positive integers $d$ and $a_1, \dots, a_n$ with $1\leq a_1\leq \dots \leq a_n\leq d$.
The ideal of Veronese type of $S$ indexed by $d$ and $(a_1, \dots,a_n)$ is the ideal $I_{d;a_1, \dots, a_n}$
which is generated by those monomials $u=x_1^{u_1}\dots x_n^{u_n}$ of $S$ of degree $d$ with $u_i\leq a_i $ for each $1\leq i\leq n$.

 The set of associated prime ideals and depth of the ideal of Veronese type are described in \cite[Proposition 5.2 and Corollary 5.7]{HRV} as follows
\begin{equation}
\label{ass}
\Ass(S/I) = \{\pp_F: F\subset [n], \sum_{i=1}^na_i\geq d-1+|F|\quad \text {and} \;\; \sum_{i\not\in F}a_i\leq d-1\},
\end{equation}
and
\begin{equation}
\label{depthver}
\depth S/I=\max\{0, d+n-1-\sum_{i=1}^na_i\}.
\end{equation}

\begin{Proposition}
\label{veronese}
The ideal of Veronese type has maximal depth if and only if there exists a $\pp_F \in \Ass(S/I)$ where $|F|=\sum_{i=1}^na_i-(d-1)$.
\end{Proposition}
\begin{proof}
In view of (\ref{ass}), $\pp_F$ has the maximum height if $|F|=\sum_{i=1}^na_i-(d-1)$. Thus $\mdepth S/I=n-|F|=n+d-1-\sum_{i=1}^na_i$ which is the same as $\depth S/I$ by (\ref{depthver}).
Therefore, the conclusion follows.

\end{proof}
Here is an example
\begin{Example}{\em
Consider $I=I_{5; 1, 2, 3}\subset S=K[x_1, x_2, x_3].$ Then $I=(x_1^2x_2^2x_3, x_1^3x_2x_3, x_1^3x_2^2)$. Formula (\ref{ass}) yields
\[
\Ass(S/I)=\{(x_1), (x_2), (x_1, x_2), (x_1, x_3), (x_2, x_3) \}.
\]
As $\pp_F \in \Ass(S/I)$ with $|F|=2$, $I$ has maximal depth and $\depth S/I=\mdepth S/I=1$.  }
\end{Example}

\section{Powers of ideals}

A subset $D\subset [n]$ is called an {\em independent set} of $G$ if $D$ contains no set $\{i,j\}$ which is an edge of $G$. The graph $G$ is called {\em bipartite} if $V(G)$ is  the disjoint union of $V_1$ and $V_2$ such that $V_1$ and $V_2$ are independent sets.
The bipartite graph $G$ is called a {\em complete bipartite} graph if $\{i, j \}\in E(G)$ for all $i\in V_1$ and $j \in V_2$.
\begin{Proposition}
\label{Complete}
Let $G$ be a complete bipartite graph on the vertex set $V$ with bipartition $V=V_1\union V_2$ where  $V_1=\{v_1,\ldots,v_n\}$ and $V_2=\{w_1,\ldots,w_m\}$ with $1\leq n\leq m$. Then $G$ has maximal depth if and only if $n=1$, i.e., $G$ is a star graph.
\end{Proposition}
\begin{proof}
The edge ideal of $G$ is $\pp_1 \cap \pp_2$ where $\pp_1=(x_1, \dots, x_n)$ and $\pp_2=(y_1, \dots, y_m)$. We set $S=K[x_1, \dots, x_n, y_1, \dots, y_m]$ and $R=S/(\pp_1 \cap \pp_2)$. Consider the exact sequence $0\to S/(\pp_1 \cap \pp_2) \to S/\pp_1 \dirsum  S/\pp_2 \to S/(\pp_1+\pp_2) \to 0.$ Since $\depth S/(\pp_1+ \pp_2)=0$, it follows from Lemma \ref{Depth lemma}(Depth lemma) that
 $\depth R=1$. On the other hand, $\Ass(R)=\{ \pp_1, \pp_2\}$. It follows that $\mdepth R=n$.
Consequently, $\mdepth R-\depth R=n-1$. Therefore, the conclusion follows.
\end{proof}
\begin{Remark}{\em
In Proposition \ref{Complete}, we showed $\mdepth R-\depth R=n-1$. Thus, the difference between $\depth$ and $\mdepth$ can be any number.
 }
\end{Remark}



In the following we classify all connected bipartite graph such that $I^k$ has maximal depth for all $k\gg 0$.

\begin{Proposition}
Let $G$ be a connected bipartite graph and $I=I(G)$ its edge ideal. Then $I^k$ has maximal depth for $k\gg 0$ if and only if $G$ is a star graph.
\end{Proposition}
\begin{proof}
Suppose $I^k$ has maximal depth for $k\gg 0$. By \cite[Corollary 10.3.18]{HH}, we have  $\depth S/I^k=1$ for $k\gg 0$. Hence $\mdepth S/I^k=1$ for $k\gg 0$. As $G$ is bipartite, we have $\Ass (I)=\Ass(I^k)$ for all $k$, see \cite{LT}. It follows that $\mdepth S/I=1$ and hence $\depth S/I=1.$
Thus there exists a minimal vertex cover $F$ of $G$ such that $|[n]\setminus F|=1$. Therefore, $G$ is a star graph.

Now, suppose $G$ is a star graph and $J$ is its edge ideal. By Proposition \ref{Complete} we have $\depth S/J=\mdepth S/J=1$. As $G$ is bipartite, we have $\mdepth S/J^k=1$ for all $k$. It follows that $\depth S/J^k=1$ for all $k$ and hence $J^k$ has maximal depth for $k\gg 0$.
\end{proof}
\begin{Remark}{\em
Let $I$ be an ideal in a Noetherian ring $R$.  Brodmann [3] showed that $\Ass(I^k)=\Ass(I^{k+1})$
for all $k\gg 0$. The ideal $I$ for which $\Ass(I^k)\subset \Ass(I^{k+1})$ for all $k\geq 1$, is said to satisfy the persistence property. Edge ideals of graphs and  polymatroidal ideals have persistence property, see \cite{HRV}, \cite{MMV}.
In this case, we have  $\mdepth S/I^{k+1}\leq \mdepth S/I^k$ for all $k$ and say the ideal $I$ has non-increasing mdepth.}
\end{Remark}
\begin{center}
{Acknowledgment}
\end{center}
\hspace*{\parindent}

I would like to thank J\"urgen Herzog for helpful discussions on this work. I would also like to thank the referee for helpful comments on this article.

\end{document}